\documentclass{amsart}

\usepackage{amsfonts,amsmath,amssymb,amsthm,mathrsfs,amsxtra}
\usepackage{enumerate,verbatim}
\usepackage[all,2cell,ps]{xy}
\usepackage{hyperref}

\DeclareMathOperator{\ann}{ann}

\DeclareMathOperator{\codim}{codim}

\DeclareMathOperator{\depth}{depth}

\DeclareMathOperator{\Ext}{Ext}

\DeclareMathOperator{\gdim}{G-dim}

\DeclareMathOperator{\Hom}{Hom}

\DeclareMathOperator{\injdim}{injdim}

\DeclareMathOperator{\projdim}{projdim}
\DeclareMathOperator{\rank}{rank}
\DeclareMathOperator{\Soc}{Soc}
\DeclareMathOperator{\Tor}{Tor}

\renewcommand{\ge}{\geqslant}
\renewcommand{\le}{\leqslant}

\theoremstyle{plain}
\newtheorem{theorem}{Theorem}[section]
\newtheorem{lemma}[theorem]{Lemma}

\newtheorem{corollary}[theorem]{Corollary}

\newtheorem{innercustomthm}{Theorem}
\newenvironment{customthm}[1]
  {\renewcommand\theinnercustomthm{#1}\innercustomthm}
  {\endinnercustomthm}
  
\theoremstyle{definition}
\newtheorem{definition}[theorem]{Definition}
\newtheorem{example}[theorem]{Example}

\newtheorem{question}[theorem]{Question}

\newtheorem{para}[theorem]{}

\theoremstyle{remark}
\newtheorem{remark}[theorem]{Remark}

\numberwithin{equation}{theorem}

\title[Criteria for regular and Gorenstein rings via syzygies]{Some criteria for regular and Gorenstein local rings via syzygy modules}

\author[Dipankar Ghosh]{Dipankar Ghosh}
\address{Chennai Mathematical Institute, H1, SIPCOT IT Park, Siruseri, Kelambakkam, Chennai 603103, Tamil Nadu, India}
\email{dghosh@cmi.ac.in}

\subjclass[2010]{Primary 13D02; Secondary 13D07, 13H05, 13H10} 
\keywords{Cohen-Macaulay; Gorenstein; Regular; Canonical module; Semidualizing module; Syzygy; Ext; Tor; Minimal multiplicity}

\begin{document}

\begin{abstract}
	Let $ R $ be a Cohen-Macaulay local ring. We prove that the $ n $th syzygy module of a maximal Cohen-Macaulay $ R $-module cannot have a semidualizing direct summand for every $ n \ge 1 $. In particular, it follows that $ R $ is Gorenstein if and only if some syzygy of a canonical module of $ R $ has a non-zero free direct summand. We also give a number of necessary and sufficient conditions for a Cohen-Macaulay local ring of minimal multiplicity to be regular or Gorenstein. These criteria are based on vanishing of certain Exts or Tors involving syzygy modules of the residue field.
\end{abstract}

\maketitle

\section{Introduction}\label{Introduction}
 
 Let $(R, \mathfrak{m}, k)$ be a commutative Noetherian local ring. Let $ M $ be a finitely generated $ R $-module. For $ n \ge 0 $, let $\Omega_n^R(M)$ be the $n$th syzygy module in a minimal free resolution of $ M $. We abbreviate Cohen-Macaulay (resp. maximal Cohen-Macaulay) to CM (resp. MCM). In the study of Hochster's Canonical Element Conjecture from the point of view of syzygies, Dutta showed the following properties of syzygies:
 
 \begin{theorem}{\rm \cite[Corollary~1.2]{Dut89}}\label{thm: Dutta free summand}
 	Let $ R $ be a $ d $-dimensional local ring. Assume that $ \depth(R) > 0 $. Let $ M $ be an $ R $-module of finite length. Then:
 	\begin{enumerate}[{\rm (i)}]
 		\item
 		If $ R $ is CM, then only $ \Omega_d^R(M) $ may have a non-zero free direct summand.
 		\item
 		If $ \depth(R) < d - 1 $, then no $ \Omega_n^R(M) $ can have a non-zero free direct summand.
 		\item
 		If $ \depth(R) = d - 1 $, then no $ \Omega_n^R(M) $ can have a non-zero free direct summand for $ n \neq d - 1 $, and $ \Omega_{d-1}^R(M) $ has no non-zero free direct summand if and only if the canonical element conjecture holds true in this case.
 	\end{enumerate}
 \end{theorem}
 
 In \cite{Gol84}, Golod introduced the notion of semidualizing module by the name of suitable module. An $R$-module $M$ is said to be {\it semidualizing} if the following hold:
 \begin{enumerate}[(i)]
 	\item The natural homomorphism $R \to \Hom_R(M,M)$ is an isomorphism.
 	\item $\Ext_R^i(M,M) = 0$ for all $i \ge 1$.
 \end{enumerate}
 For example, $R$ itself is a semidualizing $R$-module. If $R$ is a CM local ring with a canonical module $\omega$, then $\omega$ is a semidualizing $ R $-module; see, e.g., \cite[3.3.10]{BH98}. In this article, we examine whether syzygies of an arbitrary module over a CM local ring have semidualizing direct summands. Our result is motivated by the following theorem of Martsinkovsky \cite[Proposition~7]{Mar96}:
 
 \begin{theorem}[Martsinkovsky]\label{thm:Martsinkovsky}
 	If $ R $ is non-regular, then no direct sum of the syzygy modules of $k$ maps onto a non-zero $R$-module of finite projective dimension.
 \end{theorem}
 
 In \cite[Corollary~9]{Avr96}, Avramov generalized this result by showing that each non-zero homomorphic image of a finite direct sum of syzygy modules of $ k $ has maximal projective complexity and curvature. In this direction, we prove the following:
 
 \begin{customthm}{\ref{thm: syz mcm no free image}}\label{thm: I}
 	Let $ R $ be a CM local ring, and $ M $ be an MCM $ R $-module. Let $ L $ be a non-zero homomorphic image of a finite direct sum of $ \Omega_n^R(M) $, $ n \ge 1 $. Then $ L $ cannot be a semidualizing $ R $-module. In particular, $ L $ cannot be free, or $ L $ cannot be an MCM $ R $-module of finite injective dimension.
 \end{customthm}
 
 Another motivation for Theorem~\ref{thm: I} is to obtain a characterization of Gorenstein local rings via free summands of certain syzygy modules. In \cite[Corollary~1.3]{Dut89}, Dutta gave the following characterization of regular local rings.
 
 \begin{theorem}[Dutta]\label{thm: Dutta RLR}
 	The local ring $R$ is regular if and only if $\Omega_n^R(k)$ has a non-zero free direct summand for some $n \ge 0$.
 \end{theorem}
 
 Inspired by this theorem, in \cite[Theorem~2.4]{Sna10}, Snapp gave a new characterization of CM local rings by proving that $ R $ is CM if and only if for some $ n \ge 0 $, $ \Omega_n^R(R/(\underline{x})) $ has a non-zero free direct summand for some system of parameters $ \underline{x} $ that form part of a minimal set of generators for $ \mathfrak{m} $. Let $ \omega $ be a canonical module of $ R $. It is well known that $ R $ is Gorenstein if and only if $ \projdim_R(\omega) $ is finite, which is equivalent to saying that some syzygy module of $\omega$ is free. Therefore, in view of the above characterizations, one may pose the following question: If some syzygy module of $ \omega $ has a non-zero free direct summand, then what can be said about the ring? As a consequence of Theorem~\ref{thm: I}, we show that $ R $ is Gorenstein if and only if $ \Omega_n^R(\omega) $ has a non-zero free direct summand for some $ n \ge 0 $; see Corollary~\ref{cor: criteria Gor syz of can mod}.
 
 Let $ (R,\mathfrak{m},k) $ be a $ d $-dimensional CM local ring. The {\it multiplicity} of $ R $, i.e., the normalized leading coefficient of the Hilbert-Samuel polynomial $ P(n) $ ($ = $ length of $ R/\mathfrak{m}^{n+1} $ for all $ n \gg 0 $) is denoted by $ e(R) $. In \cite[(1)]{Abh67}, Abhyankar showed that
 \begin{center}
 	$e(R) \ge \mu(\mathfrak{m}) - d + 1,$
 \end{center}
 where $ \mu(M) $ denotes the minimal number of generators of an $ R $-module $ M $. If equality holds, then $ R $ is said to have {\it minimal multiplicity}, or {\it maximal embedding dimension}. It is well known that if $ k $ is infinite, then $ R $ has minimal multiplicity if and only if there exists an $ R $-regular sequence $ \underline{x} $ such that $ \mathfrak{m}^2 = (\underline{x})\mathfrak{m} $; see, e.g., \cite[4.5.14(c)]{BH98}. Investigations of these rings were started by Sally in \cite{Sal77} and \cite[Theorem~1]{Sal79}, where properties of the associated graded rings, Hilbert functions and Poincar\'{e} series were studied.
 Note that every regular local ring is CM of minimal multiplicity. But the converse is not true in general, e.g., $ R_1 = k[X,Y]/(X^2, XY, Y^2) $ and $ R_2 = k[[X,Y]]/(XY) $, where $ X $ and $ Y $ are indeterminates, and $ k $ is a field. Note that $ R_1 $ is not even Gorenstein. In this article, we provide a number of necessary and sufficient conditions for a CM local ring of minimal multiplicity to be regular or Gorenstein in terms of vanishing of certain Exts or Tors involving syzygy modules of the residue field.
 
 It can be noticed that Theorem~\ref{thm: Dutta RLR} of Dutta is a special case of Theorem~\ref{thm:Martsinkovsky} of Martsinkovsky. In \cite[Theorem~4.3]{Tak06}, Takahashi generalized Dutta's result in another direction. He showed that $R$ is regular if and only if $\Omega_n^R(k)$ has a semidualizing direct summand for some $n \ge 0$. Motivated by these results, in \cite[Corollaries~3.2 and 3.4]{GGP}, the author along with Gupta and Puthenpurakal proved that if a finite direct sum of syzygy modules of $k$ maps onto a semidualizing $ R $-module or an MCM (which is assumed to  be non-zero) $ R $-module of finite injective dimension, then $R$ is regular. In this direction, we prove the following:
 
 \begin{customthm}{\ref{thm: criteria for RLR}}\label{thm: II}
 	Let $ (R,\mathfrak{m},k) $ be a $ d $-dimensional CM local ring of minimal multiplicity. Let $ M $ and $ N $ be MCM homomorphic images of finite direct sums of syzygy modules of $ k $. {\rm(}Possibly, $ M = N ${\rm)}. Then the following statements are equivalent:
 	\begin{enumerate}[{\rm(i)}]
 		\item
 		$ R $ is regular;
 		\item
 		$ \Ext_R^i(M,N) = 0 $ for some $ (d + 1) $ consecutive values of $ i \ge 1 $;
 		\item
 		$ \Tor_i^R(M,N) = 0 $ for some $ (d + 1) $ consecutive values of $ i \ge 1 $.
 	\end{enumerate}
 \end{customthm}

 The motivation for considering consecutive vanishing of Exts as well as Tors in Theorem~\ref{thm: II} came from the remarkable result due to Avramov and Buchweitz \cite[Theorem~III]{AB00}: If $ R $ is a local complete intersection ring, then for $ R $-modules $ M $ and $N$, the following conditions are equivalent:
 \begin{enumerate}[(i)]
 	\item $ \Ext_R^i(M, N) = 0 $ for some $ (\codim(R) + 1) $ consecutive values of $ i > \dim(R) $;
 	\item $ \Tor_i^R(M,N) = 0 $ for some $ (\codim(R) + 1) $ consecutive values of $ i > \dim(R) $;
 	\item $ \Ext_R^i(M, N) = 0 = \Tor_i^R(M, N) $ for all $ i > \dim(R) $.
 \end{enumerate}
 
 Let $ (R,\mathfrak{m},k) $ be a $ d $-dimensional CM local ring. Ulrich \cite[Theorem~3.1]{Ulr84} gave a criterion for a CM local ring to be Gorenstein: If there exists an $ R $-module $ L $ with positive rank such that $ 2 \mu(L) > e(R) \rank(L) $ and $ \Ext_R^i(L,R) = 0 $ for all $ 1 \le i \le d $, then $ R $ is Gorenstein. In \cite[Theorems~2.5 and 3.4]{HH05}, Hanes and Huneke showed a few variations of this result. Recently, Jorgensen and Leuschke gave some analogous criteria in \cite[Theorems~2.2 and 2.4]{JL07}.
 
 \begin{theorem}\cite[2.2]{JL07}\label{thm: Jorgensen and Leuschke}
 	If $ R $ has a canonical module $ \omega $, and $ L $ is a CM $ R $-module such that for some $ (R \oplus L) $-regular sequence $ \underline{x} $ of length $ \dim(L) $,
 	\begin{enumerate}[{\rm (i)}]
 		\item $ \lambda_R(\mathfrak{m} L/\underline{x}L) < \mu(L) $ {\rm (}where $ \lambda_R(-) $ is the length function{\rm )} \;and
 		\item $ \Ext_R^i(L, R) = 0  $\quad for \;$ 1 + \dim(R) - \dim(L) \le i \le \dim(R) + \mu(\omega) $,
 	\end{enumerate}
 	then $ R $ is Gorenstein.
 \end{theorem}
 
 Inspired by these results, we prove the following theorem, which gives some criterion for Gorenstein local rings via syzygy modules of the residue field.
 
 \begin{customthm}{\ref{thm: criteria for Gor by Ext}}\label{thm: III}
 	Let $ (R,\mathfrak{m},k) $ be a $ d $-dimensional CM local ring of minimal multiplicity. Let $ L $ be an MCM homomorphic image of a finite direct sum of syzygy modules of $ k $. If $ \Ext_R^i(L,R) = 0 $ for some $ (d + 1) $ consecutive values of $ i \ge 1 $, then $ R $ is Gorenstein.
 \end{customthm}
 
 In particular, as a consequence of Theorem~\ref{thm: III}, we obtain that if $ R $ is CM local of minimal multiplicity, and if a finite direct sum of syzygy modules of $ k $ maps onto a non-zero $ R $-module $ L $ such that $ \gdim_R(L) = 0 $ (Definition\;\ref{defn:G-dim}), then $ R $ is Gorenstein; see Corollary~\ref{cor: Tak ques}. This gives an affirmative answer to a question of Takahashi (\cite[Question~6.6]{Tak06}) for CM local rings of minimal multiplicity.
 
 In the same spirit, we obtain a few necessary and sufficient conditions for a CM local ring of minimal multiplicity to be Gorenstein by using canonical modules.
 
 \begin{customthm}{\ref{thm: criteria for Gor by can mod}}\label{thm: IV}
 	Along with the hypotheses as in {\rm Theorem~\ref{thm: III}}, suppose also that $ R $ has a canonical module $ \omega $. Then the following statements are equivalent:
 	\begin{enumerate}[{\rm(i)}]
 		\item
 		$ R $ is Gorenstein;
 		\item
 		$ \Ext_R^i(\omega,L) = 0 $ for some $ (d + 1) $ consecutive values of $ i \ge 1 $;
 		\item
 		$ \Tor_i^R(\omega,L) = 0 $ for some $ (d + 1) $ consecutive values of $ i \ge 1 $.
 	\end{enumerate}
 \end{customthm}
 
 The organization of this article is as follows. In Section~\ref{Sec: Preliminaries}, we give a few preliminaries which we use in the next sections.  Syzygies of arbitrary modules over CM local rings are studied in Section~\ref{Sec: Semidual summand of syz mod}, where Theorem~\ref{thm: I} is proved. This yields a characterization of Gorenstein local rings via syzygies of canonical modules (see Corollary~\ref{cor: criteria Gor syz of can mod}). Criteria for regular local rings and some examples are given in Section~\ref{Sec: Characterization of RLRs}; while the same for Gorenstein local rings are shown in Section~\ref{Sec: Characterization of Gor rings}.
 
\section{Preliminaries}\label{Sec: Preliminaries}
 
 Throughout this article, unless otherwise specified, all rings are assumed to be commutative Noetherian local rings, and all modules are assumed to be finitely generated. Moreover, $(R, \mathfrak{m}, k)$ always denotes a local ring with the unique maximal ideal $\mathfrak{m}$ and residue field $k$. For an $ R $-module $ M $, and $ n \ge 0 $, we denote the $n$th syzygy module of $M$ by $\Omega_n^R(M)$, i.e., the image of the $ n $th differential of an augmented minimal free resolution of $ M $. The module $\Omega_n^R(M)$ depends on the choice of a minimal free resolution of $ M $, but is unique up to isomorphism.
 
 To prove our main results, one may assume without loss of generality that the residue field $k$ is infinite from the following observation:
 
 \begin{para}\label{para: res field infinite}
 	If the residue field $ k $ is finite, we use the standard trick to replace $ R $ by $ R' := R[X]_{\mathfrak{m}R[X]} $, where $ X $ is an indeterminate. Clearly, residue field of $ R' $ is $ k(X) $, which is infinite. For every $ R $-module $ M $, we set $ M' := M \otimes_R R' $. Note that $ \mathfrak{m} R' $ is the maximal ideal of $ R' $. Therefore $ k' = k \otimes_R R' = R'/\mathfrak{m} R' $ is the residue field of $ R' $. Hence we obtain the following equalities for Hilbert functions:
 	\[
	 	\lambda_R\left( \mathfrak{m}^n M/\mathfrak{m}^{n+1} M \right) = \lambda_{R'}\left( {\mathfrak{m}'}^n M'/{\mathfrak{m}'}^{n+1} M' \right) \quad \mbox{for all }n \ge 0.
 	\]
 	Therefore we have the following:
 	\begin{enumerate}[(i)]
 		\item $ e(\mathfrak{m}, M) = e(\mathfrak{m}', M') $. In particular, $ e(R) = e(R') $.
 		\item $ \dim(M) = \dim(M') $. In particular, $ \dim(R) = \dim(R') $.
 		\item $ \mu(\mathfrak{m}) = \mu(\mathfrak{m}') $.
 	\end{enumerate}
 	Recall that $ R $ is said to have minimal multiplicity if $ e(R) = \mu(\mathfrak{m}) - \dim(R) + 1 $. So $ R $ has minimal multiplicity if and only if $ R' $ has minimal multiplicity. In view of (ii) and (iii), we also get that $ R $ is regular if and only if $ R' $ is regular.
 	
 	Note that $ R \to R' $ is a flat local extension. Hence $ R \to R' $ is faithfully flat.
 	Since $ R \to R' $ is flat, for $ R $-modules $ M $ and $ N $, we have
 	\[
	 	\Ext_R^i(M,N) \otimes_R R' \cong \Ext_{R'}^i(M',N') ~\mbox{ and }~
	 	\Tor_i^R(M,N) \otimes_R R' \cong \Tor_i^{R'}(M',N')
 	\]
 	for all $ i \in \mathbb{Z} $. Since $ R \to R' $ is faithfully flat, we obtain that
 	\begin{align}
 		\Ext_R^i(M,N) = 0 \quad\mbox{if and only if }\quad\Ext_{R'}^i(M',N') = 0;\label{Ext-ring extn}\\
 		\Tor_i^R(M,N) = 0 \quad\mbox{if and only if }\quad\Tor_i^{R'}(M',N') = 0.\label{Tor-ring extn}
 	\end{align}
 	Since $ \depth(M) = \min\{ i : \Ext_R^i(k,M) \neq 0 \} $, in view of \eqref{Ext-ring extn}, we get that $ \depth(M) = \depth(M') $. Thus $ M $ is CM (resp. MCM) if and only if $ M' $ is CM (resp. MCM). Recall that an $ R $-module $ M $ is a {\it canonical module} of $ R $ if and only if 
 	\[
	 	\rank_k\left( \Ext_R^i(k,M) \right) = \delta_{i d} ,\quad\mbox{where } d = \dim(R).
	\]
	Therefore, in view of \eqref{Ext-ring extn}, we obtain that $ \omega $ is a canonical module of $ R $ if and only if $ \omega' $ is a canonical module of $ R' $. Moreover, it follows that $ R $ is Gorenstein if and only if $ R' $ is Gorenstein.
 \end{para}
 
 \begin{para}\label{para:soc-and-ann-of-syz}
 	Let $ M $ be an $ R $-module. For every $ n \ge 1 $, since $ \Omega_n^R(M) $ is a submodule of $ \mathfrak{m} F $ for some free module $ F $, one can easily obtain the following relation between the socle of the ring and the annihilator of the syzygy modules:
 	\[
 		\Soc(R) \subseteq \ann_R\left(\Omega_n^R(M)\right) \; \mbox{ for all } n \ge 1.
 	\]
 \end{para}
 
 \begin{para}
 	Let $ (R,\mathfrak{m},k) $ be a CM local ring. Let $ x \in R $ be an $ R $-regular element. It is not always true that $ e(R) = e(R/(x)) $. So, if $ R $ has minimal multiplicity, then it is not necessarily true that $ R/(x) $ has minimal multiplicity. This statement holds true if $ x $ is an $ R $-superficial element.
 	Recall that an element $ x \in \mathfrak{m} $ is called $ R $-{\it superficial} if there exists a positive integer $ c $ such that
 	\[
 		\left( \mathfrak{m}^{n+1} :_R x \right) \cap \mathfrak{m}^c = \mathfrak{m}^n \quad\mbox{for all } n \ge c.
 	\]
 	It is well known that if $ k $ is infinite, then there exists an $ R $-superficial element; see \cite[page~7]{Sal78}. If $ \depth(R) \ge 1 $, then for every $ R $-superficial element $ x $, it can be easily shown that $ x \notin \mathfrak{m}^2 $, which yields that $ \mu(\mathfrak{m}/(x)) = \mu(\mathfrak{m}) - 1 $. In \cite[Lemma\;4(1)]{Kub85}, it is shown that if $ \depth(R) \ge 1 $, then every $ R $-superficial element is also $ R $-regular. Moreover, if $ x \in R $ is both $ R $-superficial and $ R $-regular, then $ e(R) = e(R/(x)) $; see \cite[Lemma\;4(3)]{Kub85}. So we obtain the following:
 \end{para}
 
 \begin{lemma}\label{lem: mod x and min mult}
 	Let $ (R,\mathfrak{m},k) $ be a CM local ring of minimal multiplicity. Assume that $ k $ is infinite. If $ \dim(R) \ge 1 $, then we may choose an $ R $-regular element $ x \in \mathfrak{m} \smallsetminus \mathfrak{m}^2 $ such that $ R/(x) $ has minimal multiplicity.
 \end{lemma}
 
 The following lemma concerns the behaviour of consecutive vanishing of Exts and Tors after going modulo a regular element. It might be known to experts.
 
 \begin{lemma}\label{lemma: Ext Tor mod x}
 	Let $ R $ be a CM local ring of dimension $ d \ge 1 $. Suppose $ M $ and $ N $ are MCM $ R $-modules. Let $ x $ be an $ R $-regular element. Set $\overline{(-)} := (-) \otimes_R R/(x)$. Suppose $ m $ and $ n $ are positive integers. If $ \Ext_R^i(M,N) = 0 $ {\rm(}resp. $ \Tor_i^R(M,N) = 0 ${\rm )} for all $ n \le i \le n + m $, then
 	\begin{align*}
 	&\Ext_{\overline{R}}^i(\overline{M},\overline{N}) = 0 \quad\mbox{for all } n \le i \le n + m - 1\\
 	(\mbox{resp.} \quad &\Tor_i^{\overline{R}}(\overline{M},\overline{N}) = 0 \quad\mbox{for all } n + 1 \le i \le n + m).
 	\end{align*}
 \end{lemma}
 
 \begin{proof}
 	Since $ M $ and $ N $ are MCM $ R $-modules, $ x $ is regular on both $ M $ and $ N $. Considering the long exact sequence of Ext (resp. Tor) modules corresponding to the short exact sequence
 	\[
	 	0 \longrightarrow M \stackrel{x\cdot}{\longrightarrow} M \longrightarrow \overline{M} \longrightarrow 0,
 	\]
 	in view of $ \Ext_R^i(M,N) = 0 $ (resp. $ \Tor_i^R(M,N) = 0 $) for all $ n \le i \le n + m $, we obtain that
 	\begin{equation}\label{proposition: Ext Tor mod x: equation 1}
	 	\Ext_R^{i+1}(\overline{M},N) = 0  \quad (\mbox{resp. } \Tor_{i+1}^R(\overline{M},N) = 0)
 	\end{equation}
 	for all $ n \le i \le  n + m - 1 $.	Since $ x $ is both $ R $-regular and $ N $-regular, we have
 	\begin{equation}\label{proposition: Ext Tor mod x: equation 2}
 	\Ext_R^{i+1}(\overline{M},N) \cong \Ext_{\overline{R}}^i(\overline{M},\overline{N}) ~\mbox{ and }~ \Tor_i^R(\overline{M},N) \cong \Tor_i^{\overline{R}}(\overline{M},\overline{N})
 	\end{equation}
 	for all $ i \ge 0 $; see, e.g., \cite[page~140, Lemma~2]{Mat86}. The result now follows from \eqref{proposition: Ext Tor mod x: equation 1} and \eqref{proposition: Ext Tor mod x: equation 2}.
 \end{proof}
 
\section{Semidualizing summands of syzygy modules}\label{Sec: Semidual summand of syz mod}

In this section, we study the syzygy modules over CM local rings. We start by recalling a well-known fact about syzygy modules; see, e.g., \cite[1.1.5]{BH98}:

\begin{lemma}\label{lem: syz of mcm mod x}
	Let $ R $ be a local ring, and $ M $ be an $ R $-module. Suppose $ x \in R $ is regular on both $ R $ and $ M $. Set $ \overline{(-)} := (-) \otimes_R R/(x) $. Then we have
	\[ 
	\overline{\Omega_n^R(M)} \cong \Omega_n^{\overline{R}}(\overline{M}) \quad \mbox{for all }n \ge 0.
	\]
\end{lemma}

We prove our main result of this section by reducing to the following base case.

\begin{lemma}\label{lem: syz of mcm: depth 0: no free image}
	Let $ R $ be a local ring with $ \depth(R) = 0 $. Suppose $ M $ is an $ R $-module. Let $ L $ be a non-zero homomorphic image of a finite direct sum of $ \Omega_n^R(M) $, $ n \ge 1 $. Then $ L $ cannot be a semidualizing $ R $-module.
\end{lemma}

\begin{proof}
	Assume that $ f : \bigoplus_{n \in \Lambda} \left( \Omega_n^R(M) \right)^{j_n} \longrightarrow L $ is a surjective $ R $-module homomorphism, where $ \Lambda $ is a finite collection of positive integers, and $ j_n $, $ n \in \Lambda $, are positive integers. Then, in view of Section~\ref{para:soc-and-ann-of-syz}, we obtain that
	\begin{equation}\label{lem: syz of mcm: depth 0: no free image: eqn 1}
		\Soc(R) \subseteq \ann_R\left( \bigoplus_{n \in \Lambda} \big( \Omega_n^R(M) \big)^{j_n}  \right) \subseteq \ann_R(L).
	\end{equation}
	If possible, assume that $ L $ is a semidualizing $ R $-module. Then, by the definition of semidualizing modules, we have $ \Hom_R(L,L) \cong R $, which implies that $ \ann_R(L) = 0 $. So, by \eqref{lem: syz of mcm: depth 0: no free image: eqn 1}, we see that $ \Soc(R) = 0 $, and hence $ \depth(R) \ge 1 $, which contradicts the hypothesis $ \depth(R) = 0 $. Therefore $ L $ cannot be a semidualizing $ R $-module.
\end{proof}

Now we can achieve one of the main results of this article.

\begin{theorem}\label{thm: syz mcm no free image}
	Let $ R $ be a CM local ring. Suppose $ M $ is an MCM $ R $-module. Let $ L $ be a non-zero homomorphic image of a finite direct sum of $ \Omega_n^R(M) $, $ n \ge 1 $. Then $ L $ cannot be a semidualizing $ R $-module. In particular, $ L $ cannot be free, or $ L $ cannot be an MCM $ R $-module of finite injective dimension.
\end{theorem}

\begin{proof}
	Set $ d:= \dim(R) $. Since $ R $ is CM, there exists an $ R $-regular sequence $ \underline{x} = x_1,\ldots,x_d $ of length $ d $. We set $ \overline{(-)} := (-) \otimes_R R/(\underline{x}) $. So $ \overline{R} $ is an Artinian local ring. Let $	f : \bigoplus_{n \in \Lambda} \left( \Omega_n^R(M) \right)^{j_n} \longrightarrow L $ be a surjective $ R $-module homomorphism, where $ \Lambda $ is a finite collection of positive integers, and $ j_n $, $ n \in \Lambda $, are positive integers. Tensoring $ f $ with $ R/(\underline{x}) $, we get that
	\begin{equation}\label{thm: syz mcm no free image: eqn 1}
	\overline{f} : \bigoplus_{n \in \Lambda} \left( \overline{\Omega_n^R(M)} \right)^{j_n} \longrightarrow \overline{L}
	\end{equation}
	is a surjective $ \overline{R} $-module homomorphism. Since $ \underline{x} $ is $ R $-regular, and $ M $ is MCM, we have that $ \underline{x} $ is an $ M $-regular sequence. Hence, by virtue of Lemma~\ref{lem: syz of mcm mod x}, inductively, it can be deduced that
	\begin{equation}\label{thm: syz mcm no free image: eqn 2}
	\overline{\Omega_n^R(M)} \cong \Omega_n^{\overline{R}}(\overline{M}) \quad \mbox{for all } n \ge 1.
	\end{equation}
	Therefore, in view of \eqref{thm: syz mcm no free image: eqn 1} and \eqref{thm: syz mcm no free image: eqn 2}, we obtain that $ \overline{L} $ is a non-zero homomorphic image of a finite direct sum of syzygy modules $ \Omega_n^{\overline{R}}(\overline{M}) $, $ n \ge 1 $.
	
	If possible, assume that $ L $ is a semidualizing $ R $-module. Since $ \underline{x} $ is an $ R $-regular sequence, by virtue of \cite[page~68]{Gol84}, we have that $ \overline{L} $ is a semidualizing $ \overline{R} $-module, which contradicts the fact $ \depth(\overline{R})  = 0 $ as we see in Lemma~\ref{lem: syz of mcm: depth 0: no free image}. Therefore $ L $ cannot be a semidualizing $ R $-module.
	
	Since $ R $ itself is a semidualizing $ R $-module, we obtain that $ L $ cannot be free.
	
	For the last part, without loss of generality, we may assume that $ R $ is complete. Then $ R $ has a canonical module $ \omega $, say. It is well known that every MCM $ R $-module of finite injective dimension can be written as a direct sum of copies of $ \omega $; see, e.g., \cite[Corollary~21.14]{Eis95}.	Since canonical module $\omega$ is a semidualizing $ R $-module (\cite[3.3.10]{BH98}), in view of the first part, we obtain that $ L $ cannot be an MCM $ R $-module of finite injective dimension.
\end{proof}

As an immediate corollary of Theorem~\ref{thm: syz mcm no free image}, we obtain the following result.

\begin{corollary}\label{cor: syz no free summand}
	Let $ R $ be a $ d $-dimensional CM local ring. Suppose $ M $ is an $ R $-module, not necessarily MCM. Then $ \Omega_n^R(M) $ $( n \ge d + 1 )$ cannot have a non-zero direct summand $ L $ of the following types:
	\begin{enumerate}[{\rm (i)}]
		\item $ L $ is a semidualizing $ R $-module.
		\item $ L $ is a free $ R $-module.
		\item $ \injdim_R(L) $ is finite.
	\end{enumerate}
\end{corollary}

\begin{proof}
	Note that every non-zero syzygy module $ \Omega_n^R(M) $ (where $ n \ge d $) is an MCM $ R $-module; see, e.g., \cite[1.3.7]{BH98}. Since $ \Omega_d^R(M) $ is MCM, (i) and (ii) simply follows from Theorem~\ref{thm: syz mcm no free image}. For (iii), note that any non-zero direct summand of $ \Omega_n^R(M) $ (where $ n \ge d + 1 $) is also an MCM $ R $-module. Therefore (iii) also follows from Theorem~\ref{thm: syz mcm no free image}.
\end{proof}

The following elementary example shows that there exists an $ R $-module $ M $ such that for every $ 0 \le n \le d $, $ \Omega_n^R(M) $ has a non-zero free direct summand.

\begin{example}\label{exam: free summand for syz for o le n le d}
	Let $ R $ be a $ d $-dimensional CM local ring. Let $ x_1,\ldots,x_d $ be an $ R $-regular sequence. Fix $ n $ with $ 1 \le n \le d $. Recall that the Koszul complex $ K_{\bullet}(x_1,\ldots,x_n ; R) $ takes the form
	\[
		0 \longrightarrow  R \longrightarrow R^{\binom{n}{n-1}} \longrightarrow R^{\binom{n}{n-2}} \longrightarrow \cdots \longrightarrow R^{\binom{n}{1}} \longrightarrow R \longrightarrow 0.
	\]
	Moreover, since $ x_1,\ldots,x_n $ is an $ R $-regular sequence, we have that $ K_{\bullet}(x_1,\ldots,x_n ; R) $ is a minimal free resolution of $ R/(x_1,\ldots,x_n) $. Hence $ \Omega_n^R(R/(x_1,\ldots,x_n)) = R $ for every $ 1 \le n \le d $. We set
	\[
	M := R \oplus R/(x_1) \oplus R/(x_1,x_2) \oplus \cdots \oplus R/(x_1,\ldots,x_d).
	\]
	Clearly, for every $ 0 \le n \le d $, $ R $ is a direct summand of $ \Omega_n^R(M) $.
\end{example}

As an application of Theorem~\ref{thm: syz mcm no free image}, we obtain the following characterization of Gorenstein local rings via syzygies of canonical modules.

\begin{corollary}\label{cor: criteria Gor syz of can mod}
	Let $ R $ be a CM local ring with canonical module $ \omega $. Then the following statements are equivalent:
	\begin{enumerate}[{\rm (i)}]
		\item $ R $ is Gorenstein;
		\item $ \Omega_n^R(\omega) $ has a non-zero free direct summand for some $ n \ge 0 $.
	\end{enumerate}
\end{corollary}

\begin{proof}
	If $ R $ is Gorenstein, then $ \Omega_0^R(\omega) = \omega \cong R $. For the other implication, suppose that $ \Omega_n^R(\omega) $ has a non-zero free direct summand for some $ n \ge 0 $. Then, by virtue of Theorem~\ref{thm: syz mcm no free image}, $ n $ must be equal to $ 0 $. Thus $ \omega $ has a non-zero free direct summand, and hence $ R $ has finite injective dimension, i.e., $ R $ is Gorenstein.
\end{proof}

\begin{remark}\label{rmk: charac Gor by syz of MCM mod of fin inj dim}
	It can be observed that in the proof of Corollary~\ref{cor: criteria Gor syz of can mod}, we only use that $ \omega $ is an MCM $ R $-module of finite injective dimension. Hence, in Corollary~\ref{cor: criteria Gor syz of can mod}, canonical module $ \omega $ of $ R $ can be replaced by an arbitrary MCM $ R $-module of finite injective dimension.
\end{remark}

\section{Criteria for regular local rings}\label{Sec: Characterization of RLRs}
 
 In this section, we obtain a few criteria for CM local rings of minimal multiplicity to be regular in terms of vanishing of certain Exts or Tors involving syzygy modules of the residue field.
 
 \begin{theorem}\label{thm: criteria for RLR}
 	Let $ (R,\mathfrak{m},k) $ be a $ d $-dimensional CM local ring of minimal multiplicity. Let $ M $ and $ N $ be MCM homomorphic images of finite direct sums of syzygy modules of $ k $. {\rm(}Possibly, $ M = N ${\rm)}. Then the following statements are equivalent:
 	\begin{enumerate}[{\rm(i)}]
 		\item
 		$ R $ is regular;
 		\item
 		$ \Ext_R^i(M,N) = 0 $ for some $ (d + 1) $ consecutive values of $ i \ge 1 $;
 		\item
 		$ \Tor_i^R(M,N) = 0 $ for some $ (d + 1) $ consecutive values of $ i \ge 1 $.
 	\end{enumerate}
 \end{theorem}
 
 \begin{proof}
 	(i) $ \Longrightarrow $ \{(ii) and (iii)\}: Let $ R $ be a regular local ring. Then, by virtue of the Auslander-Buchsbaum Formula, we obtain that every MCM $ R $-module is free. Therefore $ M $ is a free $ R $-module, and hence
 	\[
 	\Ext_R^i(M,N) = 0 = \Tor_i^R(M,N) \quad \mbox{for all }i \ge 1.
 	\]
 	
 	\{(ii) or (iii)\} $ \Longrightarrow $ (i): By the observations made in Section~\ref{para: res field infinite}, we may assume that the residue field $ k $ is infinite. We prove the implications (ii) $ \Rightarrow $ (i) and (iii) $ \Rightarrow $ (i) by using induction on $ d $. Let us first consider the base case $ d = 0 $. In this case, $ R $ has minimal multiplicity is equivalent to saying that $ \mathfrak{m}^2 = 0 $. Therefore, in view of Section~\ref{para:soc-and-ann-of-syz}, we have that $ \mathfrak{m} \subseteq \Soc(R) \subseteq \ann_R(\Omega_n^R(k)) $ for all $ n \ge 1 $. Thus $ \mathfrak{m} \, \Omega_n^R(k) = 0 $ for all $ n \ge 0 $. Since $ M $ and $ N $ are homomorphic images of finite direct sums of syzygy modules of $ k $, we obtain that $ \mathfrak{m} M = 0 $ and $ \mathfrak{m} N = 0 $. Therefore $ M $ and $ N $ are non-zero $ k $-vector spaces. So $ \Ext_R^i(M,N) = 0 $ for some $ i \ge 1 $ yields that $ \Ext_R^i(k,k) = 0 $ for some $ i \ge 1 $, which gives that $ \projdim_R(k) $ is finite, and hence $ R $ is regular. For another implication, $ \Tor_i^R(M,N) = 0 $ for some $ i \ge 1 $ yields that $ \Tor_i^R(k,k) = 0 $ for some $ i \ge 1 $, which also implies that $ \projdim_R(k) $ is finite, and hence $ R $ is regular.
 	
 	We now give the inductive step. We may assume that $ d \ge 1 $. Therefore, in view of Lemma~\ref{lem: mod x and min mult}, there exists an $ R $-regular element $ x \in \mathfrak{m} \smallsetminus \mathfrak{m}^2 $ such that $ R/(x) $ has minimal multiplicity. We set $ \overline{(-)} := (-) \otimes_R R/(x)$. So $ \overline{R} $ is a $ (d-1) $-dimensional CM local ring of minimal multiplicity. Since $ M $ and $ N $ are MCM $ R $-modules, and $ x $ is $ R $-regular, we get that $ x $ is regular on both $ M $ and $ N $. Hence $ \overline{M} $ and $ \overline{N} $ are MCM $ \overline{R} $-modules. Let
 	\[
 		f : \bigoplus_{n \in \Lambda'} \left( \Omega_n^R(k) \right)^{j'_n} \longrightarrow M \quad\mbox{ and }\quad g : \bigoplus_{n \in \Lambda''} \left( \Omega_n^R(k) \right)^{j''_n} \longrightarrow N
 	\]
 	be surjective $ R $-module homomorphisms, where $ \Lambda' $ and $ \Lambda'' $ are finite collections of non-negative integers, and $ \{ j'_n, j''_n \} $ are positive integers. Tensoring $ f $ and $ g $ with $ R/(x) $, we obtain that
 	\begin{equation}\label{thm: criteria for RLR: eqn 1}
 		\overline{f} : \bigoplus_{n \in \Lambda'} \left( \overline{\Omega_n^R(k)} \right)^{j'_n} \longrightarrow \overline{M} \quad\mbox{ and }\quad \overline{g} : \bigoplus_{n \in \Lambda''} \left( \overline{\Omega_n^R(k)} \right)^{j''_n} \longrightarrow \overline{N}
 	\end{equation}
 	are surjective $ \overline{R} $-module homomorphisms. By virtue of \cite[Corollary~5.3]{Tak06}, for all $ n \ge 0 $, we have
 	\begin{equation}\label{thm: criteria for RLR: eqn 2}
 		\overline{\Omega_n^R(k)} \cong \Omega_n^{\overline{R}}(k) \oplus \Omega_{n-1}^{\overline{R}}(k) \quad [\mbox{by setting } \Omega_{-1}^{\overline{R}}(k) = 0].
 	\end{equation}
 	So \eqref{thm: criteria for RLR: eqn 1} and \eqref{thm: criteria for RLR: eqn 2} together imply that $ \overline{M} $ and $ \overline{N} $ are MCM homomorphic images of finite direct sums of $ \Omega_n^{\overline{R}}(k) $, $ n \ge 0 $. In view of Lemma~\ref{lemma: Ext Tor mod x}, $ \Ext_R^i(M,N) = 0 $ (resp. $ \Tor_i^R(M,N) = 0 $) for some $ (d + 1) $ consecutive values of $ i \ge 1 $ yields that $ \Ext_{\overline{R}}^i(\overline{M},\overline{N}) = 0 $ (resp. $ \Tor_i^{\overline{R}}(\overline{M},\overline{N}) = 0 $) for some $ d ~ (= \dim(\overline{R}) + 1) $ consecutive values of $ i \ge 1 $. Therefore, by induction hypothesis, we obtain that $ \overline{R} $ is regular, and hence $ R $ is regular as $ x \in \mathfrak{m} \smallsetminus \mathfrak{m}^2 $ is an $ R $-regular element. This proves the implications (ii) $ \Rightarrow $ (i) and (iii) $ \Rightarrow $ (i), and hence the theorem.
 \end{proof}
 
 \begin{remark}\label{rmk: M N mcm as dir summand of syz}
 	If $ (R,\mathfrak{m},k) $ is a $ d $-dimensional CM local ring, then every non-zero syzygy module $ \Omega_n^R(k) $ $ (\mbox{where }n \ge d) $ is MCM; see, e.g., \cite[1.3.7]{BH98}. Hence any non-zero direct summand of $ \Omega_n^R(k) $ $ (n \ge d) $ is also MCM. Therefore, in Theorem~\ref{thm: criteria for RLR}, $ M $ and $ N $ can be taken as any non-zero direct summands of $ \Omega_n^R(k) $, $ n \ge d $.
 \end{remark}
 
 The following example shows that the number of consecutive vanishing of Exts or Tors in Theorem~\ref{thm: criteria for RLR} cannot be further reduced.
 
 \begin{example}\label{exam: cannot red d no of vanish of exts: RLR}
 	Let $ R = k[[X,Y]]/(XY) $, where $ k[[X,Y]] $ is a formal power series ring in two indeterminates $ X $ and $ Y $ over a field $ k $. We set $ \mathfrak{m} := (x,y) $, where $ x $ and $ y $ are the images of $ X $ and $ Y $ in $ R $ respectively. Clearly, $ (R, \mathfrak{m}, k) $ is a CM local ring of dimension $ 1 $. It can be easily seen that $ e(R) = 2 $ and $ \mu(\mathfrak{m}) = 2 $. Therefore $ R $ has minimal multiplicity. Note that we have the following direct sum decomposition:
 	\[
	 	\Omega_1^R(k) = \mathfrak{m} = (x,y) = (x) \oplus (y).
 	\]
 	We set $ M := (x) $ and $ N := (y) $. Since $ \Omega_1^R(k) $ is MCM (by Remark~\ref{rmk: M N mcm as dir summand of syz}), we obtain that $ M $ and $ N $ are MCM homomorphic images of $ \Omega_1^R(k) $. Considering the minimal free resolution of $ M $:
 	\[
	 	\cdots \stackrel{x\cdot}{\longrightarrow} R \stackrel{y\cdot}{\longrightarrow} R \stackrel{x\cdot}{\longrightarrow} R \stackrel{y\cdot}{\longrightarrow} R \longrightarrow 0,
 	\]
 	one can easily compute the following:
 	\begin{align*}
 	& \Tor_{2i+1}^R(M,M) = (x)/(x^2) \neq 0 \quad \mbox{for all } i \ge 0,\\
 	& \Tor_{2i}^R(M,M) = 0 \quad \mbox{for all } i \ge 1,\\
 	& \Ext_R^{2i}(M,M) = (x)/(x^2) \neq 0 \quad \mbox{for all } i \ge 1 \quad \mbox{and} \\
 	& \Ext_R^{2i+1}(M,M) = 0 \quad \mbox{for all } i \ge 0.
 	\end{align*}
 	According to Theorem~\ref{thm: criteria for RLR}, we need at least $ 2 $ consecutive vanishing of Exts or Tors to conclude that $ R $ is regular. In this case, $ R $ is not regular. Note that one can also compute $ \Tor_i^R(M,N) $ and $ \Ext_R^i(M,N) $ to conclude the fact.
 \end{example}
 
 \begin{example}
 	If $ M $ and $ N $ are not homomorphic images of finite direct sums of syzygy modules of $ k $, then Theorem~\ref{thm: criteria for RLR} does not hold true.
 	
 	(1) One can always take $ M = N = R $. In this case, we have that $ \Ext_R^i(M,N) = 0 = \Tor_i^R(M,N) $ for all $ i \ge 1 $. But there are CM local rings of minimal multiplicity which are not regular.
 	
 	(2) Suppose $ R $ is a CM local ring of minimal multiplicity, which is not Gorenstein. (For example, $ R $ can be taken as
 	\[
 	k[X_1,\ldots,X_n]/(X_1,\ldots,X_n)^2 \quad\mbox{for some }n \ge 2,
 	\]
 	where $ X_1,\ldots,X_n $ are indeterminates, and $ k $ is a field). Let $ \omega $ be a canonical module of $ R $. Clearly, $ \omega $ is a non-free $ R $-module. Setting $ M = N = \omega $, we have $ \Ext_R^i(M,N) = 0 $ for all $ i \ge 1 $, but $ R $ is not regular.
 \end{example}
 
 We now close this section by presenting a natural question.
 
 \begin{question}
 	Can we drop the minimal multiplicity hypothesis in Theorem~\ref{thm: criteria for RLR}?
 \end{question}
 
 Though we have not been able to get some counterexample, but we believe that if we omit this hypothesis, then Theorem~\ref{thm: criteria for RLR} does not hold true.
 
\section{Criteria for Gorenstein local rings}\label{Sec: Characterization of Gor rings}
 
 In this section, we provide several criteria for Gorenstein local rings via syzygy modules of the residue field over CM local rings of minimal multiplicity. We start with the following theorem, which is analogous to the results by Ulrich \cite[Theorem~3.1]{Ulr84}, Hanes-Huneke \cite[Theorems~2.5 and 3.4]{HH05} and Jorgensen-Leuschke \cite[Theorems~2.2 and 2.4]{JL07}.
 
 \begin{theorem}\label{thm: criteria for Gor by Ext}
 	Let $ (R,\mathfrak{m},k) $ be a $ d $-dimensional CM local ring of minimal multiplicity. Let $ L $ be an MCM homomorphic image of a finite direct sum of syzygy modules of $ k $. If $ \Ext_R^i(L,R) = 0 $ for some $ (d + 1) $ consecutive values of $ i \ge 1 $, then $ R $ is Gorenstein.
 \end{theorem}
 
 \begin{proof}
 	In view of Section~\ref{para: res field infinite}, we may assume that the residue field $ k $ is infinite. We use induction on $ d $. We first consider the base case $ d = 0 $. In this case, we have that $ L $ is a non-zero $ k $-vector space as in the proof of Theorem~\ref{thm: criteria for RLR}.	So $ \Ext_R^i(L,R) = 0 $ for some $ i \ge 1 $ yields that $ \Ext_R^i(k,R) = 0 $ for some $ i \ge 1 $, which implies that $ R $ is Gorenstein, see, e.g., \cite[Theorem~18.1]{Mat86}.
 	
 	We now give the inductive step. We may assume that $ d \ge 1 $. So, in view of Lemma~\ref{lem: mod x and min mult}, there exists an $ R $-regular element $ x \in \mathfrak{m} \smallsetminus \mathfrak{m}^2 $ such that $ R/(x) $ has minimal multiplicity. We set $ \overline{(-)} := (-) \otimes_R R/(x)$. So $ \overline{R} $ is a $ (d-1) $-dimensional CM local ring of minimal multiplicity. As in the proof of Theorem~\ref{thm: criteria for RLR}, we get that $ \overline{L} $ is an MCM homomorphic image of a finite direct sum of $ \Omega_n^{\overline{R}}(k) $, $ n \ge 0 $. Furthermore, in view of Lemma~\ref{lemma: Ext Tor mod x}, $ \Ext_R^i(L,R) = 0 $ for some $ (d + 1) $ consecutive values of $ i \ge 1 $ yields that $ \Ext_{\overline{R}}^i(\overline{L},\overline{R}) = 0 $ for some $ d ~ (= \dim(\overline{R}) + 1) $ consecutive values of $ i \ge 1 $. Therefore, by induction hypothesis, we obtain that $ \overline{R} $ is Gorenstein, and hence $ R $ is Gorenstein as $ x $ is $ R $-regular. This completes the proof of the theorem.
 \end{proof}
 
 Let us recall the notion of G-dimension, due to Auslander and Bridger \cite{AB69}.
 
 \begin{definition}\label{defn:G-dim}
 	An $ R $-module $ M $ is said to have $ \gdim_R(M) = 0 $ if the natural map $ M \to M^{**} $ is an isomorphism, and $ \Ext_R^i(M,R) = 0 = \Ext_R^i(M^*,R) $ for all $ i \ge 1 $, where $ (-)^* := \Hom_R(-,R) $.
 \end{definition} 
 
 As a consequence of Theorem~\ref{thm: criteria for Gor by Ext}, we obtain the following result.
 
 \begin{corollary}\label{cor: Tak ques}
 	Let $ (R,\mathfrak{m},k) $ be a CM local ring of minimal multiplicity. If a finite direct sum of syzygy modules of $ k $ maps onto a non-zero $ R $-module $ L $ such that $\gdim_R(L) = 0$, then $ R $ is Gorenstein.
 \end{corollary}
 
 \begin{proof}
 	In view of the Auslander-Bridger Formula (\cite[Theorem\,(4.13)(b)]{AB69}), since $ R $ is CM and $\gdim_R(L) = 0$, we obtain that $ L $ is MCM. Since $ \gdim_R(L) = 0 $, we also have $ \Ext_R^i(L,R) = 0 $ for all $ i \ge 1 $. So the result follows from Theorem~\ref{thm: criteria for Gor by Ext}.
 \end{proof}
 
 \begin{remark}
 	In \cite[6.6]{Tak06}, after proving Theorem~6.5, Takahashi raised the following question which is still open: If $ \Omega_n^R(k) $ has a non-zero direct summand of G-dimension $ 0 $ for some $ n > \depth(R) + 2 $, then is $R$ Gorenstein? Corollary~\ref{cor: Tak ques} (in particular) provides an affirmative answer to this question for CM local rings of minimal multiplicity.
 \end{remark}
 
 We now give a few criteria for Gorenstein local rings in terms of vanishing of certain Exts or Tors involving canonical modules.
 
 \begin{theorem}\label{thm: criteria for Gor by can mod}
 	Let $ (R,\mathfrak{m},k) $ be a $ d $-dimensional CM local ring of minimal multiplicity. Suppose $ R $ has a canonical module $ \omega $. Let $ L $ be an MCM homomorphic image of a finite direct sum of syzygy modules of $ k $. Then the following statements are equivalent:
 	\begin{enumerate}[{\rm(i)}]
 		\item
 		$ R $ is Gorenstein;
 		\item
 		$ \Ext_R^i(\omega,L) = 0 $ for some $ (d + 1) $ consecutive values of $ i \ge 1 $;
 		\item
 		$ \Tor_i^R(\omega,L) = 0 $ for some $ (d + 1) $ consecutive values of $ i \ge 1 $.
 	\end{enumerate}
 \end{theorem}
 
 \begin{proof}
 	(i) $ \Longrightarrow $ \{(ii) and (iii)\}: If $ R $ is Gorenstein, then $ \omega \cong R $. Hence we have
 	\[
	 	\Ext_R^i(\omega,L) = 0 = \Tor_i^R(\omega,L) \quad \mbox{for all }i \ge 1.
 	\]
 	
 	\{(ii) or (iii)\} $ \Longrightarrow $ (i): Without loss of generality, we may assume that the residue field $ k $ is infinite. As before, to prove these implications ((ii) $ \Rightarrow $ (i) and (iii) $ \Rightarrow $ (i)), we use induction on $ d $. We first assume that $ d = 0 $. In this case, we have that $ L $ is a non-zero $ k $-vector space. So $ \Ext_R^i(\omega,L) = 0 $ for some $ i \ge 1 $ yields that $ \Ext_R^i(\omega,k) = 0 $ for some $ i \ge 1 $, which gives that $ \projdim_R(\omega) $ is finite, and hence
 	$ R $ is Gorenstein. For another implication, $ \Tor_i^R(\omega,L) = 0 $ for some $ i \ge 1 $ yields that $ \Tor_i^R(\omega,k) = 0 $ for some $ i \ge 1 $, which also implies that $ \projdim_R(\omega) $ is finite, and hence $ R $ is Gorenstein.
 	
 	For the inductive step, we assume that $ d \ge 1 $. In view of Lemma~\ref{lem: mod x and min mult}, there exists an $ R $-regular element $ x \in \mathfrak{m} \smallsetminus \mathfrak{m}^2 $ such that $ R/(x) $ has minimal multiplicity. We set $ \overline{(-)} := (-) \otimes_R R/(x)$. So $ \overline{R} $ is a $ (d-1) $-dimensional CM local ring of minimal multiplicity. As before, we get that $ \overline{L} $ is an MCM homomorphic image of a finite direct sum of $ \Omega_n^{\overline{R}}(k) $, $ n \ge 0 $. Since $ \omega $ is a canonical module of $ R $, it is well known that $ \overline{\omega} $ is a canonical module of $ \overline{R} $. By virtue of Lemma~\ref{lemma: Ext Tor mod x}, $ \Ext_R^i(\omega,L) = 0 $ (resp. $ \Tor_i^R(\omega,L) = 0 $) for some $ (d + 1) $ consecutive values of $ i \ge 1 $ yields that $ \Ext_{\overline{R}}^i(\overline{\omega},\overline{L}) = 0 $ (resp. $ \Tor_i^{\overline{R}}(\overline{\omega},\overline{L}) = 0 $) for some $ d ~ (= \dim(\overline{R}) + 1) $ consecutive values of $ i \ge 1 $. Therefore, by induction hypothesis, we obtain that $ \overline{R} $ is Gorenstein, and hence $ R $ is Gorenstein. This proves the implications, and hence the theorem.
 \end{proof}
 
 \begin{remark}\label{rmk: of thm: charac of Gor by can mod}
 	It can be noticed that in the proof of Theorem~\ref{thm: criteria for Gor by can mod}, it is only used that $ \omega $ is an MCM $ R $-module of finite injective dimension. Therefore, in Theorem~\ref{thm: criteria for Gor by can mod}, one can replace $ \omega $ (canonical module of $ R $) by an arbitrary MCM $ R $-module of finite injective dimension.
 \end{remark}
 
 We now analyze a few examples.
 
 \begin{example}
 	In Theorems~\ref{thm: criteria for Gor by Ext} and \ref{thm: criteria for Gor by can mod}, the number of consecutive vanishing of Exts or Tors cannot be further reduced. Let $ (R,\mathfrak{m},k) $ be an Artinian local ring such that $ \mathfrak{m}^2 = 0 $ and $ \mu(\mathfrak{m}) \ge 2 $. (For example, $ R $ can be taken as
 	\[
	 	k[X_1,\ldots,X_n]/(X_1,\ldots,X_n)^2 \quad\mbox{for some }n \ge 2,
	\]
	where $ X_1,\ldots,X_n $ are indeterminates). Clearly, the ring $ R $ has minimal multiplicity, and it is not Gorenstein. Let $ L $ be a non-zero homomorphic image of a finite direct sum of syzygy modules of $ k $. Since $ R $ is Artinian, $ L $ is an MCM $ R $-module. In this case, we need at least one $ (= \dim(R) + 1) $ vanishing of $ \Ext_R^i(L,R) $, $ \Ext_R^i(\omega,L) $ or $ \Tor_i^R(\omega,L) $ to conclude that $ R $ is Gorenstein. So the number of consecutive vanishing of Exts or Tors cannot be further reduced.
 \end{example}
 
 \begin{example}
 	If $ L $ is not a homomorphic image of a finite direct sum of syzygy modules of $ k $, then Theorems~\ref{thm: criteria for Gor by Ext} and \ref{thm: criteria for Gor by can mod} do not hold true. For example, taking $ L = R $ in Theorems~\ref{thm: criteria for Gor by Ext} and \ref{thm: criteria for Gor by can mod}, we have that $ \Ext_R^i(L,R) = 0 = \Tor_i^R(\omega,L) $ for all $ i \ge 1 $; while $ L = \omega $ in Theorem~\ref{thm: criteria for Gor by can mod} yields that $ \Ext_R^i(\omega,L) = 0 $ for all $ i \ge 1 $. But there are CM local rings of minimal multiplicity which are not Gorenstein.
 \end{example}
 
 As before, we may now ask the following natural question:
 
 \begin{question}
 	Can we omit the hypothesis that $ R $ has minimal multiplicity in Theorems~\ref{thm: criteria for Gor by Ext} and \ref{thm: criteria for Gor by can mod}?
 \end{question}

\section*{Acknowledgements}
The author would like to express his sincere gratitude to Hailong Dao who suggested the results of Jorgensen and Leuschke. The motivation for Theorem~\ref{thm: III} came from these results. It is a pleasure to thank Clare D'Cruz, Manoj Kummini and Tony J. Puthenpurakal for fruitful discussions and giving several helpful suggestions. The author would also like to thank Chennai Mathematical Institute for providing postdoctoral fellowship for this study. Finally, the author thanks the referee for a few comments and suggestions.


\begin{thebibliography}{AAAA}
	
\bibitem[Abh67]{Abh67} S. S. Abhyankar, {\it Local rings of high embedding dimension}, Amer. J. Math. {\bf 89} (1967) 1073--1077.

\bibitem[AB69]{AB69} M. Auslander and M. Bridger, {\it Stable module theory}, Mem. Amer. Math. Soc. {\bf 94} (1969).

\bibitem[Avr96]{Avr96} L. L. Avramov, {\it Modules with extremal resolutions}, Math. Res. Lett. {\bf 3} (1996), 319--328.

\bibitem[AB00]{AB00} L. L. Avramov and R.-O. Buchweitz. {\it Support varieties and cohomology over complete intersections}, Invent. Math., {\bf 142} (2000), 285--318.

\bibitem[BH98]{BH98} W. Bruns and J. Herzog, {\it Cohen-Macaulay Rings}, Cambridge Studies in Advanced Mathematics {\bf 39}, Revised Edition, Cambridge University Press, Cambridge, 1998.

\bibitem[Dut89]{Dut89} S. P. Dutta, {\it Syzygies and homological conjectures}, in: Commutative Algebra, Berkeley, CA, 1987, in: Math. Sci. Res. Inst. Publ., vol. 15, Springer, New York, 1989, pp. 139--156.

\bibitem[Eis95]{Eis95} D. Eisenbud, {\it Commutative Algebra with a View Toward Algebraic Geometry}, Graduate Texts in Mathematics {\bf 150}, Springer-Verlag, New York, 1995.

\bibitem[GGP]{GGP} D. Ghosh, A. Gupta and T. J. Puthenpurakal, {\it Characterizations of regular local rings via syzygy modules of the residue field}, To appear in J. Commut. Algebra, \href{http://arxiv.org/pdf/1511.08012v1.pdf}{arXiv:1511.08012}.
%

\bibitem[Gol84]{Gol84} E. S. Golod, {\it G-dimension and generalized perfect ideals}, in: Algebraic Geometry and Its Applications, Trudy Mat. Inst. Steklov. {\bf 165} (1984), 62--66.

\bibitem[HH05]{HH05} D. Hanes and C. Huneke, {\it Some criteria for the Gorenstein property}, J. Pure Appl. Algebra {\bf 201} (2005), 4--16.

\bibitem[JL07]{JL07} D. A. Jorgensen and G. J. Leuschke, {\it On the growth of the Betti sequence of the canonical module}, Math. Z. {\bf 256} (2007), 647--659.

\bibitem[Kub85]{Kub85} K. Kubota, {\it On the Hilbert-Samuel function}, Tokyo J. Math. {\bf 8} (1985) 439--448.

\bibitem[Mar96]{Mar96} A. Martsinkovsky, {\it A remarkable property of the {\rm (}co{\rm )} syzygy modules of the residue field of a nonregular local ring},  J. Pure Appl. Algebra {\bf 110} (1996), 9--13.

\bibitem[Mat86]{Mat86} H. Matsumura, {\it Commutative Ring Theory}, Cambridge University Press, Cambridge, 1986.

\bibitem[Sal77]{Sal77} J. D. Sally, {\it On the associated graded ring of a local Cohen-Macaulay ring}, J. Math. Kyoto Univ. {\bf 17} (1977), 19--21.

\bibitem[Sal78]{Sal78} J. D. Sally, {\it Number of generators of ideals in local rings}, In: M. Dekker (ed.) Lect.
Notes Pure Appl. Math. {\bf 35}, 1978.

\bibitem[Sal79]{Sal79} J. D. Sally, {\it Cohen-Macaulay local rings of maximal embedding dimension}, J. Algebra {\bf 56} (1979), 168--183.

\bibitem[Sna10]{Sna10} B. Snapp, {\it Free summands of syzygies of modules over local rings}, J. Pure Appl. Algebra {\bf 214} (2010), 1808--1811. 

\bibitem[Tak06]{Tak06} R. Takahashi, {\it Syzygy modules with semidualizing or G-projective summands}, J. Algebra {\bf 295} (2006), 179--194.
 
\bibitem[Ulr84]{Ulr84} B. Ulrich, {\it Gorenstein rings and modules with high numbers of generators}, Math. Z. {\bf 188} (1984), 23--32.
 
\end{thebibliography}
\end{document}